\documentclass[11pt,oneside]{amsart}
\usepackage{amsmath}
\usepackage{amsfonts}
\usepackage{amssymb}
\usepackage{graphicx}
\usepackage[mathscr]{eucal}
\newtheorem{theorem}{Theorem}
\theoremstyle{plain}

\newtheorem*{question}{Question}

\newtheorem{lemma}{Lemma}

\theoremstyle{remark}
\newtheorem{remark}{Remark}

\usepackage[height=210mm,width=158mm]{geometry}                       
\begin{document}
\title[The Unimodality Conjecture for cubical polytopes]{The Unimodality Conjecture for cubical polytopes}

\author{L\'aszl\'o Major}
\address{L\'aszl\'o Major\newline%
\indent Demartment of Statistics,   \newline%
\indent Faculty of Social Sciences, \newline%
\indent  E\"{o}tv\"{o}s Lor\'and University, Budapest \newline%
\indent  P\'azm\'any P\'eter s\'eét\'any 1/A, H-1117, Budapest, Hungary }
\email{major@tatk.elte.hu}%

\author{Szabolcs T\'oth}
\address{Szabolcs T\'oth \newline%
\indent University of Szeged,  \newline%
\indent \'Opusztaszer, Karak\'as telep 466. H-6767, Hungary }
\email {tothsz77@postafiok.hu}%

\hspace{-4mm} \date{Dec 24, 2014}

\begin{abstract}
Although the Unimodality Conjecture holds for some certain classes of cubical polytopes (e.g. cubes, capped cubical polytopes, neighborly cubical polytopes),  it fails for cubical polytopes in general. A 12-dimensional cubical polytope with non-unimodal face vector is constructed by using capping operations over a  neighborly cubical polytope with 2 to the power 131 vertices. For cubical polytopes, the Unimodality Conjecture is proved for dimensions less than  11. The first one-third of the face vector of a cubical polytope is increasing and its last one-third is decreasing in any dimension.
\end{abstract}

\maketitle
\section{Introduction}

The vector $\textbf{a}=(a_0,\ldots, a_{d-1})$ is  called \textit{unimodal} if for some (not necessarily unique) index $i$, $(a_0,\ldots, a_i)$ is non-decreasing  and  $(a_i,\ldots, a_{d-1})$ is non-increasing. If that is the case, we say that the unimodal vector $\textbf{a}$ \textit{peaks} at $i$. We say that the vector $\textbf{a}$ \textit{dips} at $i$ if $f_{j}>f_i<f_{k}$ for some $0\leq j<i<k\leq d-1$.  Clearly, the vector $\textbf{a}$ is unimodal if and only if it does not dip anywhere. The question of unimodality of the members of certain classes of vectors  has been of  long-standing interest   in algebra, combinatorics, graph theory and geometry  (see e.g. \cite{B1,sta}). 

 By  \textit{$f$-vector} (\textit{face vector}) we mean the vector $(f_0,\ldots,f_{d-1})$, where  $f_i$ is the number of $i$-dimensional proper faces of a $d$-polytope. The unimodality  of face vectors of  polytopes is  extensively studied (see e.g. \cite{bj1,eck,maj,SZ1}).   In 1961 (according to Bj\"orner \cite{bj1}), Motzkin conjectured that the $f$-vector of any polytope is unimodal.  The Unimodality Conjecture for polytopes was also stated by Welsh \cite{wel}. Danzer already showed in 1964 (see \cite[Section 2.1]{zi1}) that the conjecture cannot stand in its full generality, still leaving open the question: which natural classes of polytopes have unimodal $f$-vectors? 
 
 Examples of simplicial polytopes with non-unimodal face vectors were first published by Bj\"orner \cite{bj2}. Bj\"orner's original counterexamples were $24$-dimensional, but  subsequently also $20$-dimensional counterexamples were constructed by Bj\"orner \cite{bj3} and independently by Lee \cite{bil,lee}. It was shown by Eckhoff \cite{eck} that, in fact, this is the smallest dimension in which simplicial counterexamples can be found. In Section \ref{non}, we construct a $12$-dimensional cubical polytope with non-unimodal face vector and we show in Section \ref{small}, that there is no cubical counterexamples of dimensions less than 11. 
 
 Bj\"orner conjectured a partial unimodality property for polytopes. Namely, the face vectors of polytopes increase on the first quarter, and they decrease on the last quarter.
Bj\"orner has proved  in \cite{bj1}, that this conjecture holds for simplicial polytopes, the face vectors of simplicial polytopes moreover  increase up to the middle, and they decrease
on the last quarter.  In Section \ref{par}, we prove a similar  statement for cubical polytopes: their face vectors can dip only on the middle one-third part.

 \section{Cubes and cubical polytopes}
 The \textit{$d$-cube}  (denote by $C^d$)  is a polytope
  combinatorially equivalent to the unit cube $[0, 1]^d$. 
It is a well-known fact that the face vector of the $d$-cube is unimodal and peaks at $\lfloor \frac{d}{3} \rfloor$ (in addition, it also peaks at $\lfloor \frac{d+1}{3} \rfloor$).  This fact can also be expressed as follows: let $j\in \{0,1,2\}$ such that $d\equiv j \mod 3$, then the $f$-vector of the $d$-cube peaks at $\frac{d-j}{3}$. If $j=2$, then it also peaks  at $\frac{d-j}{3}+1$. The  $f$-vector of the $d$-cube has no more peaks.   
 The $f$-vector of the $d$-cube is strictly increasing up to $\lfloor \frac{d}{3} \rfloor$ and it is strictly decreasing from $\lfloor \frac{d+1}{3} \rfloor$ on. That is,
\begin{equation}\label{szig}
f_0(C^d)<\cdots <f_{\lfloor \frac{d}{3} \rfloor}(C^d)\hspace{3mm} \text{and}\hspace{3mm} f_{\lfloor \frac{d+1}{3}\rfloor}(C^d)>\cdots > f_{d-1}(C^d)
\end{equation}
 A $d$-polytope is called \textit{cubical} provided all its facets are combinatorially equivalent to $(d-1)$-cubes (in other words, its all proper faces are cubes).  
The following important combinatorial invariant of a cubical polytope is introduced  by Adin \cite{adi}. 

Let $P$ be a cubical $d$-polytope with f-vector $\textbf{f}=(f_0,\ldots,f_{d-1})$ and let $H$ be the $d\times d$ matrix given by
\begin{equation}\label{fhH}
H(i,j)=2^{-j}\binom{d-i-1}{d-j-1}, \hspace{4mm}  \text{for} \hspace{3mm} 0\leq i,j \leq d-1
\end{equation}
  Define the \textit{short cubical h-vector} $\textbf{h}^{(sc)}=(h_0^{(sc)},\ldots,h_{d-1}^{(sc)})$ of $P$ by $\textbf{h}^{(sc)}=\textbf{f}\cdot H^{-1}$. 
Equivalently, the face vector of $P$ can be expressed by
\begin{equation}\label{fh}
\textbf{f}=\textbf{h}^{(sc)}\cdot H
\end{equation}
\begin{lemma}\label{egy}  Let $d$ be a positive integer and let $H$ be the matrix defined in \text{\textnormal{(\ref{fhH})}}.
\vspace{1.5mm}Then
\begin{enumerate}
\item[(\textit{i})] $H(i,i)<H(i,i+1)<\cdots <H(i, \lfloor\frac{d+2i}{3}\rfloor-1)\vspace{2mm}\leq H(i, \lfloor\frac{d+2i}{3}\rfloor)>\cdots >H(i,d-1)$ \\for $0\leq i \vspace{1.5mm}\leq d-1$
\item[(\textit{ii})]$H(i,j)=H(i+1,j)+2H(i+1,j+1)$ for $0\leq i,j \leq d-2$
\end{enumerate}
 \end{lemma}
\begin{proof}Denote the $i^{th}$ row of $H$ by $H(i,*)$. Let us note that for all $0\leq i \leq d-1$,  $H(i,*)$  can be viewed  as the concatenation of two vectors. The first one is the null vector (with $i$ components) and the second one is $2^{1-d}\cdot \textbf{f}(C^{d-i-1})$, where $\textbf{f}(C^{d-i-1})$ is the face vector (supplemented with the last component $f_{d-i-1}=1$)  of a $(d-i-1)$-dimensional cube.  Therefore, the inequalities of $(i)$ follow from (\ref{szig}). Thus, $H(i,*)$ is unimodal and peaks at $\lfloor \frac{d-i-1+1}{3} \rfloor+i=\lfloor \frac{d-i}{3} \rfloor+i=\lfloor \frac{d+2i}{3} \rfloor$ and also at $\lfloor \frac{d-i-1}{3} \rfloor+i=\lfloor \frac{d+2i-1}{3} \rfloor$ for all $0\leq i \leq d-1$. 
The recursion of $(ii)$ follows from the so-called Pascal's rule, i.e. $\binom{n}{k}= \binom{n-1}{k-1}+\binom{n-1}{k}$.
\end{proof}
 \begin{remark}\label{rem}  Alternatively, we can separate three different cases as follows: let $j\in \{0,1,2\}$ such that $d-i\equiv j \mod 3$, then $H(i,*)$ peaks at $\frac{d+2i-j}{3}$. If $j=0$, then it also peaks  at $\frac{d+2i-j}{3}-1$. The vector $H(i,*)$ has no more peaks.  \end{remark}
 The following lemma 
 can be verified through a case-by-case analysis. The proof based on the inequalities $(i)$ and the recursion $(ii)$ of Lemma \ref{egy} and some elementary properties of the binomial coefficients. We omit the details. 
Alternatively, it can also be proved by adapting the methods (with some necessary  modifications) applied by Bj\"orner in the proof of Lemma 6 and Lemma 7  of \cite{bj1}.
 \begin{lemma}\label{bi}
Let $d$ be a positive integer and $0\leq i\leq k\leq d-1$. Let $H$ be the matrix defined in \text{\textnormal{(\ref{fhH})}}. Let $a_j=H(i,j)+H(k,j)$ for $0\leq j \leq d-1$. Then
$$a_0<\cdots <a_{\lfloor\frac{d+2i}{3}\rfloor-1}\vspace{0mm}\leq a_{ \lfloor\frac{d+2i}{3}\rfloor}>\cdots >a_{d-1}$$
\end{lemma}
The following important lemma
is needed to prove Theorem \ref{thp} and Theorem \ref{kis}. Statement $(ii)$ is a brief formulation of the Cubical Dehn-Sommerville Equations.
\begin{lemma}[Adin \cite{adi}, Lemma 1, Corollary 10]\label{ad1}
Let $P$ be a cubical $d$-polytope. Then
\begin{enumerate}
\item[(\textit{i})] all the components of $\textbf{h}^{(sc)}(P)$ are positive integers,
\item[(\textit{ii})] $\textbf{h}^{(sc)}(P)$ is symmetric: $h_i^{(sc)}=h_{d-i-1}^{(sc)}$, $(0\leq i \leq d-1),$
\item[(\textit{iii})] $\textbf{h}^{(sc)}(P)$ is unimodal.
\end{enumerate}
 \vspace{-1mm}\end{lemma}
\section{Capped and neighborly cubical polytopes}\label{cnc}
 
 There is a cubical analogue of the simplicial stacking operation, the so-called \textit{capping operation} described by Jockusch \cite{joco} as follows: let $Q$ be a cubical $d$-polytope, then the polytope $P$ is called a \textit{capped polytope over} $Q$ if there is a 
$d$-cube $C$ such that $P = Q \cup C$ and $Q \cap C$ is a facet of both $Q$ and $C$. Roughly speaking, $P$ is obtained by glueing a cube onto a facet of $Q$. A polytope $P$ is said to be an $n$\textit{-fold capped polytope over} $Q$ if there is a sequence $P_0, P_1,\ldots,P_n$ ($1\leq n$) of polytopes such that for $i=0,\ldots,n-1$
\begin{enumerate}
\item[(\textit{i})] $P_{i+1}$ is a capped polytope over $P_i$,
\item[(\textit{ii})] $P_0=Q$ and $P_n=P$.
\end{enumerate}
For $n=0$, the $n$\textit-fold capped polytope over $Q$ is $Q$ itself. A polytope is said to be $n$\textit{-fold capped} (or simply \textit{capped}) if it is an $n$-fold capped polytope over a cube. Capped polytopes are the cubical analogues of the (simplicial) stacked polytopes. 

Since the capping operation destroys a cubical facet while it creates $2d-1$ new ones, the capping operation increases the component $f_k$ of the face vector  by $2^{d-k}\binom{d}{k}-2^{d-k-1}\binom{d-1}{k}$ if $0\leq k \leq d-2$ and by $2(d-1)$ if $k=d-1$. Hence, if $P$ is an  $n$-fold capped polytope over $Q$, then we have 
\begin{equation}\label{cap}
 f_k(P)=\left\{ \begin{array}{ll}
f_k(Q)+n\Big(2^{d-k}\binom{d}{k}-2^{d-k-1}\binom{d-1}{k}\Big) &
\hspace{3mm}\vspace{0mm}\textrm{if $0\leq k \leq d-2$} \\
f_k(Q)+2n(d-1) & \hspace{3mm}\textrm{if $k=d-1$} 
\end{array} \right. 
\end{equation}
By using (\ref{cap}) and the fact that the face vector of the $d$-cube is unimodal and peaks at $\lfloor \frac{d+1}{3} \rfloor$, it is not difficult to show that the face vector of a capped $d$-polytope is  also unimodal and also peaks at $\lfloor \frac{d+1}{3} \rfloor$.

The \textit{$k$-skeleton} of  a $d$-polytope is the union of its $k$-dimensional faces. A cubical $d$-polytope (with $2^n$ vertices for some $n\geq d$) is called \textit{neighborly cubical} provided its $(\lfloor \frac{d}{2} \rfloor -1)$-skeleton is combinatorially equivalent to the $(\lfloor \frac{d}{2} \rfloor -1)$-skeleton of a cube. The concept of neighborly cubical polytopes was introduced by Babson, Billera and Chan \cite{bill}. Neighborly cubical polytopes can be considered as the cubical analogues of the (simplicial) cyclic polytopes. 
It is proved in \cite{maj1}, that the Unimodality Conjecture holds for neighborly cubical polytopes. The number of vertices and the dimension of a neighborly cubical polytope determine its $f$-vector  and it is given by (see \cite{maj1})
\begin{displaymath}
 f_k=\left\{ \begin{array}{ll}
2^{n-k}\displaystyle\sum_{i=0}^{\frac{d-2}{2}} 
\textstyle
\Big(\binom{d-i-1}{k-i}+\binom{i}{k-d+i+1}\Big)\binom{n-d+i}{i} &
\textrm{if $d$ is even} \\
2^{n-k}\Bigg(\displaystyle\sum_{i=0}^{\frac{d-3}{2}} 
\textstyle
\Big(\binom{d-i-1}{k-i}+\binom{i}{k-d+i+1}\Big)\binom{n-d+i}{i}+\displaystyle\sum_{j=0}^{n-d}2^{-j}\textstyle\binom{\frac{d-1}{2}}{d-k-1}
\binom{n-\frac{d+3}{2}-j}{n-d-j}\Bigg) & \textrm{if $d$ is odd} 
\end{array} \right. 
\end{displaymath} 
By using the above explicit formula, it can be shown that the $f$ -vector of a neighborly cubical polytope peaks approximately at $\lfloor \frac{2d}{3} \rfloor$ if $n$ is large \vspace{2mm} enough.

\section{Partial unimodality for cubical polytopes}\label{par}

In this section, we show that  the maximal component of the face vector of a cubical polytope can occur only in the middle one-third part (i.e. between $\lfloor \frac{d}{3} \rfloor$ and $\lfloor \frac{2d}{3} \rfloor$), furthermore, the violation of the Unimodality Conjecture is possible only in this part. 
\begin{theorem}[partial unimodality for cubical polytopes]\label{thp}
Let $P$ be a cubical $d$-polytope with face vector $(f_0,\ldots,f_{d-1})$. Then 
\begin{enumerate}
\item[(\textit{i})]$f_0<\cdots <f_{\lfloor \frac{d}{3} \rfloor-1}\leq f_{\lfloor \frac{d}{3} \vspace{1mm}\rfloor}$ 
\item[(\textit{ii})] $f_{\lfloor \frac{2d}{3} \rfloor}>\cdots>f_{d-1}$
\end{enumerate}
\end{theorem}
\begin{proof}Let $H$ be the matrix defined in \text{\textnormal{(\ref{fhH})}}. For $0\leq i\leq {d-1}$,  denote the $i^{th}$ row of $H$ by $H(i,*)$. For $0\leq i\leq \lfloor \frac{d-1}{2}\rfloor$, let us define the vectors $\textbf{b}^i$ by
\begin{displaymath}
 \textbf{b}^i=\left\{ \begin{array}{ll}
H(i,*)+H(d-i-1,*) &
\hspace{3mm}\textrm{if $2i\neq d-1$} \vspace{2mm}\\
H(i,*) & \hspace{3mm}\textrm{if $2i=d-1$} 
\end{array} \right. 
\end{displaymath}
By using the above notation and the symmetric property of the short $h$-vector (see $(ii)$ of Lemma \ref{ad1}), the relation (\ref{fh}) can be rewritten  as 
$$\textbf{f}(P)=\sum^{\lfloor\frac{d-1}{2}\rfloor}_{i=0}h^{(sc)}_i\textbf{b}^i$$ 
From  Lemma \ref{bi} follows that $\textbf{b}^i$ is unimodal and peaks at $\lfloor \frac{d+2i}{3} \rfloor$ for all $0\leq i \leq \lfloor\frac{d-1}{2} \rfloor$  Furthermore, we have
$$\textbf{b}^i_i<\cdots <\textbf{b}^i_{\lfloor\frac{d+2i}{3}\rfloor-1}\vspace{0mm}\leq \textbf{b}^i_{ \lfloor\frac{d+2i}{3}\rfloor}>\cdots >\textbf{b}^i_{d-1}$$
Therefore, all the vectors $\textbf{b}^i$ peak between $\lfloor \frac{d+2\cdot 0}{3} \rfloor=\lfloor \frac{d}{3} \rfloor$ and $\lfloor \frac{d+2\lfloor \frac{d}{2} \rfloor}{3} \rfloor=\lfloor \frac{2d}{3} \rfloor$ and

$$\textbf{b}^i_0\leq\cdots \leq\textbf{b}^i_{\lfloor\frac{d}{3}\rfloor-1}\vspace{0mm}\leq \textbf{b}^i_{ \lfloor\frac{d}{3}\rfloor} \text{ and }\textbf{b}^i_{ \lfloor\frac{2d}{3}\rfloor}>\cdots >\textbf{b}^i_{d-1}$$

for all $0\leq i\leq \lfloor \frac{d-1}{2}\rfloor$. Furthermore, $\textbf{b}^i_0<\cdots <\textbf{b}^i_{\lfloor\frac{d}{3}\rfloor-1}\vspace{0mm}\leq \textbf{b}^i_{ \lfloor\frac{d}{3}\rfloor}$ for  $i=0$. Consequently, $\textbf{f}(P)$  has the stated property, since $\textbf{f}(P)$ is a non-negative linear combination of the vectors  $\textbf{b}^i$ (see $(i)$ of Lemma \ref{ad1}). 
\end{proof}
\begin{remark}\label{r2} In fact,  we could state a little bit more about $\textbf{f}(P)$. Namely, it can be easily checked that if $d\equiv j \mod 6 $ for some $j\in \{0,2,3\}$, then 
 $\textbf{b}^{\lfloor \frac{d}{2} \rfloor}$ peaks at $\lfloor \frac{2d}{3} \rfloor-1$. Consequently, the sequence $(f_{\lfloor \frac{2d}{3} \rfloor-1},f_{\lfloor \frac{2d}{3} \rfloor},\ldots,f_{d-1})$ is strictly decreasing if the dimension of $P$ is congruent to $0$ or $2$ or $3$ modulo $6$.\end{remark}
 
\section{Non-unimodal $f$-vectors}\label{non}
According to Theorem \ref{thp},  capped polytopes and  neighborly cubical polytopes are extremal among all cubical polytopes, in the sense that the maximal component of their $f$-vectors occur at the most far away position from their middle. Since the peaks of the $f$-vectors of a capped and a neighborly cubical polytope are situated  as far away from each other as possible, it seems to be reasonable to involve these polytopes in constructing counterexamples to the Unimodality Conjecture for cubical polytopes. In fact, most of the  non-cubical counterexamples were also based on this idea. For instance, Danzer used stacked polytopes (with peaks at $\lfloor \frac{d}{2} \rfloor$) and crosspolytopes (with peaks at $\lfloor \frac{2d}{3} \rfloor$ ) for his first counterexamples  (according to Ziegler  \cite[Section 2.1]{zi1}).  

According to Theorem \ref{thp} and Remark \ref{r2}, for all cubical $12$-polytopes,
$$f_0<\cdots <f_3\leq f_4 \hspace{2mm} \text{and} \hspace{2mm}  f_7>\cdots >f_{11}$$ 
Therefore, the possible positions for a dip are $5$ and $6$. The starting point of our construction is a neighborly cubical $12$-polytope, to which we apply the capping operation. Due to (\ref{cap}), by applying the capping operation, 
 we are adding a vector  with peak at $\lfloor \frac{d}{3} \rfloor$
to a vector whose peak is at $\lfloor \frac{2d}{3} \rfloor$. 

   Joswig and Ziegler proved in \cite{jos} that there exists a $d$-dimensional neighborly cubical polytope  with $2^n$ vertices for any $n\geq d\geq 2$. They  constructed neighborly cubical polytopes as linear projections of cubes. For the base of our counterexample, we chose a neighborly cubical $12$-polytope with $2^{131}$ vertices, to which we apply the capping operation $1.841\cdot 10^{42}$ times. Then we obtain a cubical polytope with non-unimodal $f$-vector.
By using (\ref{cap}) and the formula of the $f$-vector of neighborly cubical polytopes (see in Section \ref{cnc}), it can be computed, that    the $f$-vector of this polytope dips  at $5$ indeed.

\begin{theorem}
There exists a  $12$-dimensional cubical polytope with $3.770370722\cdot 10^{45}$ 
vertices for which $f_4 >f_5 < f_6$. Consequently, the Unimodality Conjecture fails for  cubical polytopes in general.
\end{theorem}


As a matter of historical interest, we mention that Lee's simlicial counterexample was one of the first  applications of  the sufficiency part of the famous $g$-theorem  (see  Billera and Lee, Corollary  2 in \cite{bil}).  

\section{Unimodality for small dimensional cubical polytopes}\label{small}
Using the relations between the $f$-vectors and the $g$-vectors of simplicial polytopes, Eckhoff \cite{eck} has shown that the $f$-vector of a simplicial polytope is unimodal if its dimension is less than $20$. We prove the following statement in a similar way.
\begin{theorem}\label{kis}
The $f$-vectors of cubical $d$-polytopes are unimodal for all $d\leq 10$.
\end{theorem}
  \begin{proof}
  First let $P$ be a cubical $10$-polytope with $f$-vector $\textbf{f}=(f_0,\ldots,f_9)$. It follows from Theorem \ref{thp}, that $f_0<f_1< f_2\leq f_3$ and $f_6>\ldots> f_9$. Hence, $\textbf{f}$ can possibly dip at $4$ and $5$. Using Lemma \ref{ad1} and the relation $\textbf{f}=\textbf{h}^{(sc)}\cdot H$, one can show that all the assumptions $$f_3>f_4<f_5,\hspace{1mm} f_3>f_4<f_6, \hspace{1mm}f_3>f_5<f_6 \hspace{1mm}\text{ and }\hspace{1mm} f_4>f_5<f_6$$ lead to contradictions. Consequently, the face vector of $P$ is unimodal indeed.
  For $d<10$, similar reasoning completes the proof.
  \end{proof}
  
  The method of the above proof does not lead to an analogous result in the case $d=11$. However, we could construct some symmetric unimodal vector $\textbf{v}$ (whose components are positive integers), such that
  $v\cdot H$ would be non-unimodal, we could not guarantee  that there would exist some 11-polytope, whose short cubical $h$-vector would equal $\textbf{v}$. Since the complete combinatorial  characterization of cubical polytopes is not known, the following question still remains open:
\begin{question}Is there any cubical $11$-polytope with non-unimodal face vector?
\end{question}

To violate the unimodality in dimension $d=11$, for the role  of the $h$-vector, we need a candidate with an ``outlandish shape'', namely, its middle component should be relatively large (compared to other components). Hence, by considering the characterization of simplicial polytopes, one may believe that the answer to the above question is negative.

\end{document}